\documentclass[11pt,a4paper]{article}

\usepackage{amsmath, amsfonts, amssymb}

\def\be#1\ee{\begin{equation}#1\end{equation}}

\newtheorem{thm}{Theorem}[section]
\newtheorem{lem}[thm]{Lemma}

\newtheorem{example}[thm]{Example}

\def\P{{\mathbb{P}}}
\def\R{\mathbb{R}}
\def\E{\mathbb{E}\,}

\def\V{\mathbb{V}}

\newenvironment{proof}[1][] {\noindent {\bf Proof#1:} }{\hspace*{\fill}$\square$\medskip\par}
\newcommand{\ind}{1\hspace{-0.098cm}\mathrm{l}}

\newcommand{\eps}{\varepsilon}

\def\dd{{\, \mbox{d}}}

\def \=L{\ {\buildrel\hbox{\scriptsize d }\over =}\ }

\let\BFseries\bfseries\def\bfseries{\BFseries\mathversion{bold}} 

\title{Large deviations for infinite weighted sums of stretched exponential random variables}

\author{Frank Aurzada}

\begin{document} 
\maketitle

\begin{abstract}
We study the large deviation probabilities of infinite weighted sums of independent random variables that have stretched exponential tails. This generalizes Kiesel and Stadtm\"uller \cite{kieselstadtmueller}, who study the same objects under the assumption of finite exponential moments, and Gantert et al.\ \cite{gantertramananrembart}, who study finite weighted sums with stretched exponential tails.
\end{abstract}



\section{Introduction}
A classical result in probability theory is Cram\'er's theorem for the large deviations of sums of independent, identically distributed random variables: If $(X_i)$ is an i.i.d.\ sequence with zero mean and for some $t>0$ the moment generating function $\phi(t):=\E e^{t X_1 }$ is finite then
$$
\lim_{n\to\infty} \frac{1}{n} \log \P( \frac{1}{n}\sum_{i=1}^n X_i > x ) = - \sup_{t\in\R} ( t x  - \log \phi(t) ),\qquad x>0.
$$
It is also classical that Cram\'er's theorem can be extended to a full large deviation principle; and it can be seen as the starting point of large deviation theory, see e.g.\ \cite{dembozeitouni,deuschelstroock}.

Whenever the random variables $(X_i)$ do not have any finite exponential moment, the behaviour of the large deviations is different. This is due to the fact that then the large deviation event is produced by only one variable being unusually large. The classical result here (cf.\ \cite{nagaev1}) is as follows: if $(X_i)$ is an i.i.d.\ sequence with stretched exponential tail, $\log  \P( X_1  > t ) \sim - \kappa t^r$, as $t\to\infty$ for some $0<r<1$, and finite expectation then 
\begin{equation} \label{eqn:nagaev}
\lim_{n\to\infty} \frac{1}{n^r} \log \P( \frac{1}{n}\sum_{i=1}^n X_i > x ) = - \kappa (x-\E[X_1 ])^r,\qquad x>\E[X_1 ].
\end{equation}

In this paper, we study \emph{weighted} sums of i.i.d.\ random variables. There is quite some literature on large devations of weighted sums and their applications. The most recent general reference is Kiesel and Stadtm\"uller \cite{kieselstadtmueller} (also see \cite{bonin1,bonin2,book2,book1,deltuvienesaulis,giulianomacci,giulianomacci2,nagaev2} for further references). However, these papers deal with random variables that do have some finite exponential moment.

The only source, to the knowledge of the author, that deals with weighted sums of random variables that do \emph{not} have any finite exponential moment is Gantert et al.\ \cite{gantertramananrembart}. There, {\it finite} sums of the type $\sum_{i=1}^n a_i(n) X_i$ are considered when the random variables have stretched exponential tails.

In this note, we treat the case of \emph{infinite} weighted sums $\sum_{i=1}^\infty a_i(n) X_i$ with $(X_i)$ i.i.d.\ random variables having stretched exponential tails. Besides filling this gap in the literature, the motivation comes from Baysian statistics: There, one is interested in proving contraction rates for the posterior distribution for nonparametric inverse problems. There, estimates of the type studied here are important, see e.g. Lemma~5.2 in \cite{knapiksalomond}, \cite{ray}, or \cite{gugushvilivandervaartyan} for results with Gaussian priors, which require large deviation estimates of squared Gaussians, i.e.\ with exponential moments. We mention that the present results are directly motivated by a forthcoming work of S.~Agapiou and P.\ Math\'e in that area for non-Gaussian priors.


The paper is structured as follows. In Section~\ref{sec:mainresult}, we define the concrete setup for this paper and state our main result. The proofs are given in Section~\ref{sec:proofs}.

\section{Main result} \label{sec:mainresult}
Let $(a_i(n))_{i\geq 1, n=1,2,\ldots}$ be an array of non-negative numbers (let $\sup_i a_i(n)>0$ for all $n$ to avoid trivialities). Let $(X_i)$ be a sequence of non-negative i.i.d.\ random variables, copies of the random variable $X$ with tail behaviour
\begin{equation} \label{eqn:tailone}
\log \P( X  > t ) \sim -\kappa t^{r},\qquad \text{as $t\to\infty$,}
\end{equation}
for some $0<r<1$ and $\kappa>0$.

We are interested in the probability
\begin{equation} \label{eqn:finitepcase}
\P\left(  \sum_{i=1}^\infty a_i(n) X_i  > x  \right),\qquad \text{where $x>0$ and $n\to\infty$.}
\end{equation}
The large deviation regime is characterized by the condition that the typical values of $\sum_{i=1}^\infty a_i(n) X_i$ lie below $x$, i.e.\
$$
\limsup_{n\to\infty}\E\left[ \sum_{i=1}^\infty a_i(n)  X_i  \right]    < x,
$$
which we shall encode using assumption (\ref{eqn:pfinupnatcondNEW}) below. 

We can now formulate our main result, which is a ``largest jump principle'' for the large deviations of weighted sums of stretched exponential random variables. This means that the large deviation event is triggered by one of the terms in the sum being large, namely the one corresponding to the largest weight.

\begin{thm} \label{thm:upperboundpfinite}  Let $(X_i)$ be a sequence of non-negative i.i.d.\ random variables, copies of $X$ with tail behaviour (\ref{eqn:tailone}). Further, let $(a_i(n))_{i\geq 1,n=1,2,\ldots}$ be non-negative numbers with
\begin{equation}
\label{eqn:pfinupnatcondNEW}
\lim_{n\to\infty} \sum_{i=1}^\infty a_i(n)  = D\in[0,\infty)
\end{equation}
and such that $a_{\max}(n):=\max_{i\geq 1} a_i(n)>0$ and $a_{\max}(n) \to 0$. 
Then for any $x> D\cdot \E[X ]$
$$
\lim_{n\to\infty} a_{\max}(n)^r\log \P\left( \sum_{i=1}^\infty a_i(n)  X_i  > x \right) = - \kappa(x-D \cdot\E [X ])^r.
$$
\end{thm}

We stress that we do not need \emph{any} regularity assumption on the sequence $(a_i(n))$. Note that $\max_{i\geq 1} a_i(n)$ exists (for any $n$), because (\ref{eqn:pfinupnatcondNEW}) implies that $\lim_{i\to\infty}a_i(n)\to 0$.

\begin{example}
The classical  result (\ref{eqn:nagaev}) is retrieved for $a_i(n)=n^{-1} \ind_{i\leq n}$.
\end{example}

\begin{example} \label{exa:sergios} In a motivating example from Baysian statistics, $a_i(n)=\ind_{i\geq n}\sigma_i/\rho_n$, which gives the large deviation probability of ``remainder'' sums: $\P( \sum_{i=n}^\infty \sigma_i X_i> x \rho_n)$. Here $(\sigma_i)$ is a positive, summable sequence and $(\rho_n)$ is a positive sequence.
\end{example}

\begin{example}
The work of Gantert et al.\ \cite{gantertramananrembart} in the case of non-negative random variables can be recovered as follows. They consider arrays with $a_i(n)=0$ for $i>n$. Their condition (B) implies that (\ref{eqn:pfinupnatcondNEW}) holds with $D>0$ and that $n \cdot a_{\max}(n) \to s>0$. Note that we do not require $a_{\max}(n)$ to be of order $n^{-1}$ in this work.
\end{example}

\begin{example} Examples where $a_i(n)$ depends on $n$ in a different way are given for instance by moving averages, where
$$
a_i(n):=  \sigma_i \phi_n^{-1} \ind_{m_n \leq i \leq m_n+\phi_n-1},
$$
for positive sequences $\sigma$, $\phi$, $m$. Such objects were studied by \cite{kieselstadtmueller} under the assumption of finite exponential moments (cf.\ the remark on p.\ 938 in \cite{kieselstadtmueller}).
\end{example}

Possible extensions of the present results include the case that $X$ has a polynomial tail (rather than stretched exponential) or the precise behaviour for the case of a supremum rather than a sum (see Lemma~\ref{lem:pinftyupperbound} below for a partial result). In the spirit of Example~\ref{exa:sergios}, one could also consider $\sum_{i=N}^\infty$, where $N$ is random (cf.\ e.g.\ \cite{bonin2} for the case of finite sums). Further, one might want to add a slowly varying factor in (\ref{eqn:tailone}).

\section{Proofs}\label{sec:proofs}
%

\subsection{Auxiliary results for maxima}
We start with two results for the rate of the probability
\begin{equation} \label{eqn:pinfty}
\P( \sup_{i\geq 1} a_i(n) X_i > x), \qquad n\to\infty,
\end{equation}
which is the obvious analog of (\ref{eqn:finitepcase}). We start with a lower bound.
\begin{lem} \label{lem:lowerboundinfty}
If $a_{\max}(n)\to 0$ then for any $x>0$
$$
\liminf_{n\to\infty}  a_{\max}(n)^{r} \log \P( \sup_{i\geq 1} a_i(n) X_i > x ) \geq - \kappa x^r.
$$
If $\limsup_{n\to\infty} a_{\max}(n)> 0$ then $\liminf_{n\to\infty}\P( \sup_{i\geq 1} a_i(n) X_i > x )> 0$.
\end{lem}

\begin{proof}
The claims follow immediately from the trivial estimate
$$
\P( \sup_{i\geq 1} a_i(n) X_i > x ) \geq \P( a_{\max}(n) X_{m(n)} > x ) = \P( X  > x / a_{\max}(n)).
$$
where $m(n):=\min\{ i \geq 1 : a_i(n)=a_{\max}(n)\}$.
\end{proof}

We now turn to the corresponding upper bound. We shall prove it under more restrictive assumptions in order to avoid lengthy discussions (note that (\ref{eqn:pfinupnatcondNEW}) is not necessary for the sup-problem). The stated lemma will be one ingredient in the proof of the main result.

\begin{lem} \label{lem:pinftyupperbound} Assume that (\ref{eqn:pfinupnatcondNEW}) holds and that $a_{\max}(n)\to 0$. Then we have for any $x>0$
$$
\lim_{n\to\infty} a_{\max}(n)^{r} \log \P( \sup_{i\geq 1} a_i(n) X_i > x ) = - \kappa x^r.
$$
\end{lem}

\begin{proof} First note that we can assume w.l.o.g.\ that $x=1$, as otherwise it can be absorbed as a constant factor into the sequence $(a_i(n))$.
The lower bound already follows from Lemma~\ref{lem:lowerboundinfty}. For the upper bound, observe that
\begin{equation}
\P(  \sup_{i\geq 1} a_i(n) X_i > 1 ) 
= \P( \bigcup_{i=1}^\infty \lbrace a_i(n) X_i > 1 \rbrace)
\leq\sum_{i=1}^\infty \P( a_i(n) X_i > 1 ). \label{eqn:lastegmd}
\end{equation}

Fix $0<\eps<\kappa$. It remains to use the tail bound for $X$, which shows that the last term bounded from above as follows: For large enough $n$,
\begin{equation} \label{eqn:treatsumstandard}
\sum_{i=1}^\infty \P( a_i(n) X_i > 1 ) = \sum_{i=1}^\infty \P( X  > 1/a_i(n) ) \leq \sum_{i=1}^\infty C e^{-(\kappa-\eps) a_i(n)^{-r} },
\end{equation}
with some constant $C>0$. The remainder of the proof consists in a treatment of this sum:
\begin{eqnarray*}
 \sum_{i=1}^\infty e^{-(\kappa-\eps) a_i(n)^{-r} }
&=&
 \sum_{i=1}^\infty e^{-(1-\eps)(\kappa-\eps) a_i(n)^{-r} }\cdot e^{-\eps(\kappa-\eps)  a_i(n)^{-r} } \notag
	\\
	&\leq& e^{-(1-\eps)(\kappa-\eps) a_{\max}(n)^{-r} }\cdot   \sum_{i=1}^\infty   (\eps(\kappa-\eps)  a_i(n)^{-r})^{-2/r} \notag
		\\
	&=& e^{-(1-\eps)(\kappa-\eps) a_{\max}(n)^{-r} }\cdot   (\eps(\kappa-\eps))^{-2/r} \sum_{i=1}^\infty   a_i(n)^{2}  \notag
			\\
	&\leq& e^{-(1-\eps)(\kappa-\eps) a_{\max}(n)^{-r} }\cdot   (\eps(\kappa-\eps))^{-2/r} a_{\max}(n) \sum_{i=1}^\infty   a_i(n) \notag
	\\
	&\leq& e^{-(1-\eps)(\kappa-\eps) a_{\max}(n)^{-r} }\cdot   (\eps(\kappa-\eps))^{-2/r} (D+\eps),
\end{eqnarray*}
where we used in the second step that $e^{-x}\leq x^{-2/r}$ for large $x$ and, in the last step, the assumptions that $\sum_{i=1}^\infty   a_i(n) \to D$ and $a_{\max}(n)\to 0$. Combining this with (\ref{eqn:lastegmd}) and (\ref{eqn:treatsumstandard}) shows
$$
\log \P(  \sup_{i\geq 1} a_i(n) X_i > 1 )  \leq -(1-\eps)(\kappa-\eps) a_{\max}(n)^{-r} + \log  [C(\eps(\kappa-\eps))^{-2/r} (D+\eps)].
$$
Multiplying by $a_{\max}(n)^r$, taking first $n\to \infty$ and then $\eps\to 0$ shows the upper bound in the statement.
\end{proof}

\subsection{Proof of the main result}
Here, we give the proofs of the lower and upper bound in Theorem~\ref{thm:upperboundpfinite}, respectively.

\begin{proof}[ of the lower bound] Throughout, we use the notation $m(n):=\min\{ i \geq 1 : a_i(n)=a_{\max}(n)\}$.

Let us first treat the case that $D=0$. Then the lower bound already follows from assumption (\ref{eqn:tailone}) together with
$$
\P( \sum_{i=1}^\infty a_i(n) X_i  > x ) \geq \P( a_{\max}(n) X_{m(n)}   > x ) = \P( X   > x/a_{\max}(n) ).
$$

Assume $D>0$. Then we can fix an $\eps>0$ with $\eps<D$. We begin by noting that
\begin{multline}
\P( \sum_{i=1}^\infty a_i(n) X_i  > x ) \geq \P( a_{\max}(n) X_{m(n)}  > x - \sum_{i=1,i\neq m(n)}^\infty a_i(n) \E[ X ] (1-\eps) ) 
\\
 \cdot~ \P( \sum_{i=1,i\neq m(n)}^\infty a_i(n) X_i > \sum_{i=1,i\neq m(n)}^\infty a_i(n) \E[X ] (1-\eps) ). \label{eqn:newlower1}
\end{multline}
Since $a_{\max}(n)\to 0$, (\ref{eqn:pfinupnatcondNEW}) implies $\sum_{i=1,i\neq m(n)}^\infty a_i(n) = \sum_{i=1}^\infty a_i(n) - a_{\max}(n)\to D$. Therefore, the first term on the right-hand side of (\ref{eqn:newlower1}), by (\ref{eqn:tailone}), satisfies
\begin{eqnarray*}
&& \liminf_{n\to\infty} a_{\max}(n)^{r} \log \P( a_{\max}(n) X_{m(n)}  > x - \sum_{i=1,i\neq m(n)}^\infty a_i(n) \E [X ] (1-\eps) )
\\
&\geq &\liminf_{n\to\infty} a_{\max}(n)^{r} \log \P( X  > (x - D \E [X ] (1-\eps)^2)/a_{\max}(n))
\\
&\geq& - \kappa (x - D \E [X ] (1-\eps)^2 )^r.
\end{eqnarray*}
We will show that the second term on the right-hand side of (\ref{eqn:newlower1}) tends to one for fixed $\eps$ and $n\to\infty$. Combining this with the last formula will finish the proof of the lower bound in the theorem.

Note that, for large enough $n$,
\begin{eqnarray*}
&& \P( \sum_{i=1,i\neq m(n)}^\infty a_i(n) X_i > \sum_{i=1,i\neq m(n)}^\infty a_i(n) \E[X ] (1-\eps) )
\\
&=& \P( \sum_{i=1,i\neq m(n)}^\infty a_i(n) (X_i-\E[X_i]) >  -\eps \E[X ] \sum_{i=1,i\neq m(n)}^\infty a_i(n) )
\\
&\geq & \P( \sum_{i=1,i\neq m(n)}^\infty a_i(n) (X_i-\E[X_i]) >  -\eps \E[X ] (D-\eps)).
\end{eqnarray*}
The last term tends to one, since by Chebyshev's inequality

\begin{eqnarray*}
&& \P( \sum_{i=1,i\neq m(n)}^\infty a_i(n) (X_i-\E[X_i]) \leq  -\eps \E[X ] (D-\eps))
\\
&\leq&\P\left( \left|\sum_{i=1,i\neq m(n)}^\infty a_i(n) (X_i-\E[X_i])\right| \geq  \eps \E[X ] (D-\eps)\right)
\\
&\leq& (\eps \E[X ] (D-\eps))^{-2} \cdot\V\left[ \sum_{i=1,i\neq m(n)}^\infty a_i(n) (X_i-\E[X_i]) \right]
\\
&=& (\eps \E[X ] (D-\eps))^{-2} \cdot\V[ X] \cdot\sum_{i=1,i\neq m(n)}^\infty a_i(n)^2
\\
&\leq& (\eps \E[X ] (D-\eps))^{-2} \cdot \V[X] \cdot a_{\max}(n) \cdot \sum_{i=1}^\infty a_i(n),
\end{eqnarray*}
which tends to zero (because the sum is bounded, by (\ref{eqn:pfinupnatcondNEW}), and $a_{\max}(n)\to 0$), as required.
\end{proof}

\begin{proof}[ of the upper bound]
The first observation is that we can assume w.l.o.g.\ that $x=1$, as $x$ can be absorbed as a constant factor into the sequence $(a_i(n))$.

{\it Step 1: Reduction step, main argument, overview.}

Set $A:=x-D \E[X ]=1-D \E[X ]$ and note that $A>0$, by assumption. Further, fix $0<\eps<\kappa/2$ such that also $1-(1+\eps) D \E[X ] > 0$. First note that
\begin{eqnarray*}
 \P(  \sum_{i=1}^\infty a_i(n) X_i  > 1 )
&\leq& \P( \sum_{i=1}^\infty a_i(n) X_i  > 1, \sup_{i\geq 1} a_i(n) X_i \leq A )
\\ && \qquad + \,\P( \sup_{i\geq 1} a_i(n) X_i > A ),
\end{eqnarray*}
and the second term can be treated with Lemma~\ref{lem:pinftyupperbound}, which shows that the second term has asymptotic order $\exp(-\kappa a_{\max}(n)^{-r} A^r (1+o(1)))$, as required by the assertion. If we can show that the first term is of the same or lower order, we obtain the statement.

{\it Step 2: Exponential Chebychev inequality for the truncated random variables.}

Let us consider the first term: For any $\lambda>0$, by the Markov inequality,
{\allowdisplaybreaks
\begin{eqnarray}
&& \P( \sum_{i=1}^\infty a_i(n) X_i  > 1, \sup_{i\geq 1} a_i(n) X_i \leq A  )
\notag
\\
& =& \P( e^{\lambda\sum_{i=1}^\infty a_i(n) X_i}  > e^{\lambda} , \sup_{i\geq 1} a_i(n) X_i \leq A   )
\notag
\\
& \leq & e^{-\lambda} \E[ e^{\lambda\sum_{i=1}^\infty a_i(n) X_i}  , \sup_{i\geq 1} a_i(n) X_i \leq A  ]
\notag
\\
& = & e^{-\lambda} \prod_{i=1}^\infty \E[ e^{\lambda a_i(n) X }  \ind_{a_i(n) X \leq A}   ]
\notag
\\
&=&
\exp\left( -\lambda +\sum_{i=1}^\infty \log \E[ e^{\lambda a_i(n) X } \ind_{a_i(n) X \leq A} ]\right) \notag
\\
&\leq&
\exp\left( -\lambda +\sum_{i=1}^\infty \left( \E[ e^{\lambda a_i(n) X } \ind_{a_i(n) X \leq A}] -  1 \right)\right) \notag
\\
&\leq&
\exp\left( -\lambda +\sum_{i=1}^\infty  \E[ (e^{\lambda a_i(n) X } -  1) \ind_{a_i(n) X \leq A} ]\right). \label{eqn:infinitecomputation1}
\end{eqnarray}
}

Let us deal with the sum. Note that for $0\leq x\leq \eps$ we have $e^x-1\leq \frac{e^{\eps}-1}{\eps} x\leq (1+\eps)x$ (for $\eps$ small enough). Thus

{\allowdisplaybreaks
\begin{eqnarray}
&&
\sum_{i=1}^\infty  \E[ (e^{\lambda a_i(n) X } -  1) \ind_{a_i(n) X \leq A} ]  \notag
\\
&=&
\sum_{i=1}^\infty  \E[ (e^{\lambda a_i(n) X } -  1) \ind_{a_i(n) X \leq A, \lambda a_i(n) X  < \eps} ] + \sum_{i=1}^\infty  \E[ (e^{\lambda a_i(n) X } -  1) \ind_{a_i(n) X \leq A, \lambda a_i(n) X  \geq \eps} ]  \notag
\\
&\leq &
\sum_{i=1}^\infty  \E[ (1+\eps)\lambda a_i(n) X   \ind_{a_i(n) X \leq A, \lambda a_i(n) X  < \eps} ] + \sum_{i=1}^\infty  \E[ (e^{\lambda a_i(n) X }-1)  \ind_{a_i(n) X \leq A, \lambda a_i(n) X  \geq \eps} ]  \notag
\\
&\leq &
(1+\eps) \lambda \, \E[ X  ]\sum_{i=1}^\infty a_i(n) + \sum_{i=1}^\infty  \E[ (e^{\lambda a_i(n) X }-1)  \ind_{a_i(n) X \leq A, \lambda a_i(n) X  \geq \eps} ].  \label{eqn:infinitecomputation2}
\end{eqnarray}
}

Setting $B:=\kappa-2\eps$ we shall use the last estimate with
$$
\lambda:=\frac{B A^{r-1}}{a_{\max}(n)^r}.
$$

{\it Step 3: We show that the second sum in (\ref{eqn:infinitecomputation2}) tends to zero for fixed $\eps>0$ and $n\to\infty$.}

First note that if $a_i(n) X \leq A$ then -- using $r<1$ -- we have
$$
\lambda a_i(n) X  = \frac{B A^{r-1} a_i(n)}{a_{\max}(n)^r} \cdot X ^{1-r} \cdot X ^r \leq \frac{B A^{r-1} a_i(n)}{ a_{\max}(n)^r}\cdot \frac{A^{1-r}}{a_i(n)^{1-r}}\cdot X ^r = \frac{B a_i(n)^r}{a_{\max}(n)^r}\cdot X ^r.
$$
Therefore,
\begin{equation} \label{eqn:tbdelemenatr}
\E[ (e^{\lambda a_i(n) X }-1)  \ind_{a_i(n) X \leq A, \lambda a_i(n) X  \geq \eps} ] \leq \E[ (e^{\frac{B a_i(n)^r}{a_{\max}(n)^r}\cdot X ^r} -1) \cdot \ind_{\lambda a_i(n) X  \geq \eps} ].
\end{equation}

Further, it is elementary to show  (see Lemma~\ref{lem:elementary} below) that due to the tail estimate (\ref{eqn:tailone}), which we use in the form $\P(X >t)\leq k \exp( - B' t^r)$ for all $t>0$ and some $k>0$, where $B':=\kappa-\eps$, we have
$$
\E[ ( e^{b X ^r}-1) \ind_{X >a} ] \leq \frac{k}{1-b/B'} e^{-(B'-b)a^r},
$$
for any $a,b>0$ with $b<B'$.

In our case, $b:=B a_i(n)^r/a_{\max}(n)^r\leq B< B'$ and $a:=\eps a_i(n)^{-1} \lambda^{-1}$. Therefore, we see that the term on the right-hand side of (\ref{eqn:tbdelemenatr}) is  bounded from above by
\begin{eqnarray*}
&& \frac{k}{1-\frac{(\kappa-2\eps)a_i(n)^r}{(\kappa-\eps)a_{\max}(n)^r}}\, \exp(- (\kappa-\eps-(\kappa-2\eps) a_i(n)^r/a_{\max}(n)^r) \cdot [ \eps a_i(n)^{-1} \lambda^{-1}]^r )
\\
&\leq& k\,\frac{\kappa-\eps}{\eps} \, \exp(- \eps \cdot B^{-r} \eps^r A^{r(1-r)} [a_i(n)^{-1} a_{\max}(n)^r]^r ).
\end{eqnarray*}
The second sum in (\ref{eqn:infinitecomputation2}) is therefore  bounded from above by
$$
c_\eps \sum_{i=1}^\infty e^{- 2K [a_i(n)^{-1} a_{\max}(n)^r]^r },
$$
where $2K=2K(\eps):=\eps^{1+r} B^{-r} A^{r(1-r)}$ and $c_\eps:=k(\kappa-\eps)/\eps$. This can be treated as follows: Since $e^{-x}\leq x^{-1/r}$ for large enough $x$, we have
\begin{eqnarray*}
\sum_{i=1}^\infty e^{- 2K [a_i(n)^{-1} a_{\max}(n)^r]^r }
&=&\sum_{i=1}^\infty e^{- K [a_i(n)^{-1} a_{\max}(n)^r]^r } \cdot e^{- K [a_i(n)^{-1} a_{\max}(n)^r]^r } 
\\
&\leq &\sum_{i=1}^\infty \left( K [a_i(n)^{-1} a_{\max}(n)^r]^r \right)^{-1/r} \cdot e^{- K a_{\max}(n)^{-(1-r)r}} 
\\
&= & K^{-1/r} e^{- K a_{\max}(n)^{-(1-r)r} } a_{\max}(n)^{-r} \sum_{i=1}^\infty a_i(n).
\end{eqnarray*}
Now, $\sum_{i=1}^\infty a_i(n)$ is bounded by assumption (\ref{eqn:pfinupnatcondNEW}). Further, since $a_{\max}(n)\to 0$, the term $e^{- K a_{\max}(n)^{-(1-r)r} }  a_{\max}(n)^{-r}$ tends to zero for fixed $\eps$ and $n\to\infty$. This finishes the proof of the fact that the second sum in (\ref{eqn:infinitecomputation2}) tends to zero.

{\it Step 4: Final computations.}
Putting Step 3 together with  (\ref{eqn:infinitecomputation1}) and (\ref{eqn:infinitecomputation2}), 
we have seen that for fixed $\eps$ and $n\to\infty$
\begin{eqnarray*}
&&\log \P( \sum_{i=1}^\infty a_i(n) X_i  > 1, \sup_{i\geq 1} a_i(n) X_i \leq A   )
\\
&\leq& - \lambda + (1+\eps) \lambda \E[ X ] \sum_{i=1}^\infty a_i(n) +o(1).
\\
&=& - \frac{B A^{r-1}}{a_{\max}(n)^r }\left[ 1  - (1+\eps) \E[ X ] \sum_{i=1}^\infty a_i(n)  - o(1) \right].
\end{eqnarray*}

Multiplying by $a_{\max}(n)^r$ and using  (\ref{eqn:pfinupnatcondNEW}), we obtain
\begin{eqnarray*}
 && \limsup_{n\to\infty}  a_{\max}(n)^r   \log \P( \sum_{i=1}^\infty a_i(n) X_i  > 1, \sup_{i\geq 1} a_i(n) X_i \leq A   )
\\
&\leq& - B A^{r-1} ( 1- (1+\eps)\E[X ] D)   = -(\kappa-2\eps) A^{r-1} ( 1- (1+\eps)\E[X ] D).
\end{eqnarray*}
Letting $\eps\to 0$ shows the assertion.
\end{proof}

During the course of the last proof, we used the following completely elementary lemma.

\begin{lem} \label{lem:elementary}
Let $X$ be a non-negative random variable with $\P(X>t)\leq k e^{-B't^r}$ for all $t>0$ and $k,B',r>0$. Then, for any $a>0$ and any $0<b<B'$,
$$
\E[ (e^{bX^r}-1) \ind_{X>a} ] \leq \frac{k}{1-b/B'} e^{-(B'-b)a^r}.
$$
\end{lem}

\begin{samepage}
\begin{proof}
Note that
\begin{eqnarray*}
\E[ (e^{bX^r}-1) \ind_{X>a} ] &=& \E[ \int_1^{e^{bX^r}} \dd s \ind_{X>a} ]
\\
&=&  \int_1^\infty\E[ \ind_{(b^{-1}\log s)^{1/r}< X } \ind_{X>a} ] \dd s
\\
&=&  \int_1^{e^{b a^r}} \P(X>a)  \dd s+\int_{e^{b a^r}}^\infty \P(X>(b^{-1}\log s)^{1/r})  \dd s
\\
&\leq& k\left( \int_0^{e^{b a^r}} e^{-B'a^r}  \dd s+\int_{e^{b a^r}}^\infty e^{-B' b^{-1} \log s}  \dd s\right)
\\
&=& k\left( e^{-(B'-b)a^r}  + \frac{1}{B'/b-1} e^{(1-B' b^{-1})b a^r}\right)
\\
&=& ke^{-(B'-b)a^r} \frac{1}{1-b/B'}.
\end{eqnarray*}
\end{proof}
\end{samepage}

\noindent\textbf{Acknowledgement.} The author is indebted to Sergios Agapiou (University of Cyprus) and Peter Math\'e (WIAS Berlin) for bringing this problem to his attention and to Marvin Kettner (Darmstadt) for valuable suggestions.

\bibliographystyle{abbrv}

\end{document}